\documentclass{birkjour}
 \newtheorem{thm}{Theorem}[section]
 \newtheorem{cor}[thm]{Corollary}
 \newtheorem{lem}[thm]{Lemma}
 \newtheorem{prop}[thm]{Proposition}
 \theoremstyle{definition}
 \newtheorem{defn}[thm]{Definition}
 \theoremstyle{remark}
 \newtheorem{rem}[thm]{Remark}
 
 \newtheorem{problem}[thm]{Problem}

 \numberwithin{equation}{section}
\usepackage{prettyref}
\usepackage{mathtools}
\usepackage{amsmath}
\usepackage{amsthm}
\usepackage{amssymb}
\usepackage{enumerate}

\newcommand{\spt}{\operatorname{spt}}

\newcommand*{\e}{\mathrm{e}}
\renewcommand*{\i}{\mathrm{i}}
\newcommand{\m}{\operatorname{m}}
\renewcommand{\tilde}{\widetilde}

\renewcommand{\Re}{\operatorname{Re}}
\renewcommand{\Im}{\operatorname{Im}}
\DeclareMathOperator*{\wlim}{w-lim}
\DeclareMathOperator{\ind}{ind}

\begin{document}

%
%
%
%
%
%
%
%
%

\title
 {On higher index differential-algebraic equations in infinite dimensions}

\author{Sascha Trostorff}

\address{%
Insitut f\"ur Analysis\\
Fakult\"at Mathematik\\
Technische Universit\"at Dresden\\
Germany}

\email{sascha.trostorff@tu-dresden.de}

\author{Marcus Waurick}
\address{Department of Mathematics and Statistics\\
University of Strathclyde\\
Glasgow, United Kingdom}
\email{marcus.waurick@strath.ac.uk}
\subjclass{Primary: 34A09, Secondary: 34A12, 34A30, 34G10}

\keywords{Differential-algebraic equations, higher index, infinite dimensional state space, consistent initial values, distributional solutions}

\date{October 23, 2017}

\begin{abstract}
We consider initial value problems for differential-algebraic equations in a possibly infinite-dimensional Hilbert space. Assuming a growth condition for the associated operator pencil, we prove existence and uniqueness of solutions for arbitrary initial values in a distributional sense. Moreover, we construct a nested sequence of subspaces for initial values in order to obtain classical solutions.
\end{abstract}

\maketitle

\section{Introduction and main results}

In this short note, we consider two solution concepts of differential-algebraic equations (DAEs) in infinite dimensions. For this, let $E$ and
$A$ be bounded linear operators in some possibly infinite dimensional
Hilbert space $H$.

We consider the implicit initial value problem
\begin{equation}
\tag{{\ensuremath{\ast}}}\begin{cases}
Eu'(t)+Au(t)=0, & t>0,\\
u(0+)=u_{0}
\end{cases}\label{eq:IVP0}
\end{equation}
for some given $u_{0}\in H$. In order to talk about a well-defined
problem in \prettyref{eq:IVP0}, we assume that the pair $(E,A)$
is \emph{regular}, that is, 
\begin{align*}
\exists\nu\in\mathbb{R}\colon & \mathbb{C}_{\Re>\nu}\subseteq \rho(E,A),\\
\exists C\geq0,k\in\mathbb{N}\:\forall s\in\mathbb{C}_{\Re>\nu}\colon & \|\left(sE+A\right)^{-1}\|\leq C|s|^{k},
\end{align*}
where 
\[
 \rho(E,A)\coloneqq \{s\in \mathbb{C}\,;\, (sE+A)^{-1}\in L(H)\}.
\]

We note here that these two conditions are our replacements for regularity
in finite dimensions. Indeed, for $H$ finite-dimensional, $(E,A)$
is called regular, if $\det(sE+A)\neq0$ for some $s\in\mathbb{C}$.
Thus, $s\mapsto\det(sE+A)$ is a polynomial of degree at most $\dim H$,
which is not identically zero. The growth condition is a consequence
of the Weierstrass or Jordan normal form theorem valid for finite
spatial dimensions, see e.g.~\cite{Berger2012,Dai1989,Mehrmann2006}.
The smallest possible $k\in\mathbb{N}$ occurring in the resolvent
estimate is called the \emph{index of $(E,A)$}:
\[
\ind(E,A)\coloneqq\min\{k\in\mathbb{N}\,;\,\exists C\geq0\,\forall s\in\mathbb{C}_{\Re>\nu}:\|\left(sE+A\right)^{-1}\|\leq C|s|^{k}\}.
\]
We shall also define a sequence of (initial value) spaces associated with $(E,A)$:
\begin{equation*}
\mathrm{IV}_{0} \coloneqq H\mbox{ and } \mathrm{IV}_{k+1} \coloneqq\{x\in H;Ax\in E[\mathrm{IV}_{k}]\}\quad(k\in\mathbb{N}).
\end{equation*}

A first observation is the following.
\begin{prop}
\label{prop:stabInde}Let $k=\ind(E,A)$ and assume that $E[\mathrm{IV}_{k}]\subseteq H$
is closed. Then $\mathrm{IV}_{k+1}=\mathrm{IV}_{k+2}$.
\end{prop}

Since the sequence of spaces $(\mathrm{IV}_{k})_{k}$ is decreasing
(see \prettyref{lem:aux}), Proposition \ref{prop:stabInde} leads to the
following question.
\begin{problem}
\label{prob:indexProblem}Assume that $E[\mathrm{IV}_{j}]\subseteq H$
is closed for each $j\in\mathbb{N}$. Do we then have
\[
\min\{k\in\mathbb{N};\mathrm{IV}_{k+1}=\mathrm{IV}_{k+2}\}=\ind(E,A)?
\]
\end{problem}

With the spaces $(\mathrm{IV}_{k})_{k}$ at hand, we can present the
main theorem of this article.
\begin{thm}
\label{thm:mainTh}Assume that $E[\mathrm{IV}_{\ind(E,A)}]\subseteq H$
is closed, $u_{0}\in\mathrm{IV}_{\ind(E,A)+1}$. Then there exists a unique continuously differentiable
function $u\colon\mathbb{R}_{>0}\to H$ with $u(0+)=u_{0}$ such that
\[
Eu'(t)+Au(t)=0\quad(t>0).
\]
\end{thm}

With Proposition \ref{prop:stabInde} and \prettyref{thm:mainTh}, it is possible to derive the following consequence.

\begin{cor}\label{cor:mainCor}Assume that $E[\mathrm{IV}_{j}]\subseteq H$
is closed for each $j\in\mathbb{N}$, $u_0\in H$. Then there exists a continuously differentiable function 
 $u\colon\mathbb{R}_{>0}\to H$ with $u(0+)=u_{0}$ and
\[
Eu'(t)+Au(t)=0\quad(t>0),
\]
if, and only if, $u_{0}\in\mathrm{IV}_{\ind(E,A)+1}$.
\end{cor}

Corollary \ref{cor:mainCor} suggests that the answer to Problem \ref{prob:indexProblem}
is in the affirmative for $H$ being finite-dimensional. 

Also in our main result, there is room for improvement: In applications,
it is easier to show that $R(E)\subseteq H$ is closed as the $\mathrm{IV}$-spaces
are not straightforward to compute. Thus, we ask whether the latter
theorem can be improved in the following way. 
\begin{problem}
Does $R(E)\subseteq H$ closed imply the closedness of $E[\mathrm{IV}_{\ind(E,A)}]\subseteq H$ or even closedness of $E[\mathrm{IV}_{j}]\subseteq H$ for all $j\in\mathbb{N}$?
\end{problem}

We shall briefly comment on the organization of this article. In the
next section, we introduce the time-derivative operator in a suitably
weighted vector-valued $L_{2}$-space. This has been used intensively
in the framework of so-called `evolutionary equations', see \cite{PicPhy}.
With this notion, it is possible to obtain a \emph{distributional}
solution of \prettyref{eq:IVP0} such that the differential algebraic
equation holds in an integrated sense, where the number of integrations
needed corresponds to the index of the DAE. We conclude this article
with the proofs of Proposition \ref{prop:stabInde}, \prettyref{thm:mainTh}, and Corollary \ref{cor:mainCor}.
We emphasize that we do not employ any Weierstrass or Jordan normal
theory in the proofs of our main results. We address the case
of unbounded $A$ to future research. The case of index $0$ is discussed
in \cite{Trostorff2017b}, where also exponential stability and dichotomies
are studied.

\section{The time derivative and weak solutions of DAEs}

Throughout this section, we assume that $H$ is a Hilbert space and
that $E,A\in L(H)$ with $(E,A)$ regular. We start out with the definition
of the space of (equivalence classes of) vector-valued $L_{2}$ functions:
Let $\nu\in\mathbb{R}$. Then we set
\[
L_{2,\nu}(\mathbb{R};H)\coloneqq\left\{ f:\mathbb{R}\to H\,;\,f\mbox{ measurable,\,}\intop_{\mathbb{R}}|f(t)|_{H}^{2}\exp(-2\nu t)\,\mathrm{d}t<\infty\right\} ,
\]
see also \cite{PicPhy,KPSTW14_OD,Picard1989}. Note that $L_{2,0}(\mathbb{R};H)=L_{2}(\mathbb{R};H)$.
We define $H_{\nu}^{1}(\mathbb{R};H)$ to be the ($H$-valued) Sobolev
space of $L_{2,\nu}(\mathbb{R};H)$-functions with weak derivative
representable as $L_{2,\nu}(\mathbb{R};H)$-function. With this,
we can define the derivative operator 
\[
\partial_{0,\nu}\colon H_{\nu}^{1}(\mathbb{R};H)\subseteq L_{2,\nu}(\mathbb{R};H)\to L_{2,\nu}(\mathbb{R};H),\phi\mapsto\phi'.
\]
In the next theorem we recall some properties of the operator just
defined. For this, we introduce the \emph{Fourier\textendash Laplace
transformation} $\mathcal{L}_{\nu}\colon L_{2,\nu}(\mathbb{R};H)\to L_{2}(\mathbb{R};H)$
as being the unitary extension of
\[
\mathcal{L}_{\nu}\phi(t)\coloneqq\frac{1}{\sqrt{2\pi}}\int_{\mathbb{R}}\phi(s)\e^{-\left(\i t+\nu\right)s}\mathrm{d}s\quad(\phi\in C_{c}(\mathbb{R};H),\,t\in\mathbb{R}),
\]
where $C_{c}(\mathbb{R};H)$ denotes the space of compactly supported,
continuous $H$-valued functions defined on $\mathbb{R}$. Moreover,
let 
\begin{align*}
\m\colon\{f\in L_{2}(\mathbb{R};H);(t\mapsto tf(t))\in L_{2}(\mathbb{R};H)\}\subseteq L_{2}(\mathbb{R};H) & \to L_{2}(\mathbb{R};H),\\
f & \mapsto(t\mapsto tf(t))
\end{align*}
be the multiplication by the argument operator with maximal domain.
\begin{thm}[{\cite[Corollary 2.5]{KPSTW14_OD}}]
\label{thm:FLT}Let $\nu\in\mathbb{R}$. Then
\[
\partial_{0,\nu}=\mathcal{L}_{\nu}^{\ast}(\i\m+\nu)\mathcal{L}_{\nu}.
\]
\end{thm}

\begin{rem}
 A direct consequence of \prettyref{thm:FLT} is the continuous invertibility of $\partial_{0,\nu}$ if $\nu\ne 0$.
\end{rem}

\begin{cor}
\label{cor:solOp}Let $\nu>0$ be such that $\rho(E,A)\supseteq\mathbb{C}_{\Re>\nu}$
and $\|\left(sE+A\right)^{-1}\|\leq C|s|^{\ind(E,A)}$ for some $C\geq0$
and all $s\in\mathbb{C}_{\Re>\nu}$. Then
\[
\partial_{0,\nu}^{-k}\left(\partial_{0,\nu}E+A\right)^{-1}\in L(L_{2,\nu}(\mathbb{R};H)),
\]
where $k=\ind(E,A)$. Moreover, $\partial_{0,\nu}^{-k}\left(\partial_{0,\nu}E+A\right)^{-1}$
is causal, i.e., for each $f\in L_{2,\nu}(\mathbb{R};H)$ with $\spt f\subseteq\mathbb{R}_{\geq a}$
for some $a\in\mathbb{R}$ it follows that 
\[
\spt\partial_{0,\nu}^{-k}\left(\partial_{0,\nu}E+A\right)^{-1}f\subseteq\mathbb{R}_{\geq a}.
\]
\end{cor}

\begin{proof}
By \prettyref{thm:FLT} and the unitarity of $\mathcal{L}_{\nu}$,
we obtain that the first claim is equivalent to
\[
\left(\i\m+\nu\right)^{-k}\left(\left(\i\m+\nu\right)E+A\right)^{-1}\in L(L_{2}(\mathbb{R};H)),
\]
which, in turn, would be implied by the fact that the function
\[
t\mapsto\left(\i t+\nu\right)^{-k}\left(\left(\i t+\nu\right)E+A\right)^{-1}
\]
belongs to $L^\infty(\mathbb{R};L(H))$. This is, however, true by regularity of $(E,A)$. We now show the
causality. As the operator $\partial_{0,\nu}^{-k}\left(\partial_{0,\nu}E+A\right)^{-1}$
commutes with translation in time, it suffices to prove the claim
for $a=0.$ So let $f\in L_{2,\nu}(\mathbb{R};H)$ with $\spt f\subseteq\mathbb{R}_{\geq0}.$
By a Paley-Wiener type result (see e.g. \cite[19.2 Theorem]{rudin1987real}),
the latter is equivalent to 
\[
(\mathbb{C}_{\Re>\nu}\ni z\mapsto\left(\mathcal{L}_{\Re z}f\right)(\Im z))\in\mathcal{H}^{2}(\mathbb{C}_{\Re>\nu};H),
\]
where $\mathcal{H}^{2}(\mathbb{C}_{\Re>\nu};H)$ denotes the
Hardy-space of $H$-valued functions on the half-plane $\mathbb{C}_{\Re>\nu}$. As 
\[
\left(\mathcal{L}_{\Re z}\partial_{0,\nu}^{-k}\left(\partial_{0,\nu}E+A\right)^{-1}f\right)(\Im z)=z^{-k}\left(zE+A\right)^{-1}\left(\mathcal{L}_{\Re z}f\right)(\Im z)
\]
for each $z\in\mathbb{C}_{\Re>\nu},$ we infer that also 
\[
(\mathbb{C}_{\Re>\nu}\ni z\mapsto\left(\mathcal{L}_{\Re z}\partial_{0,\nu}^{-k}\left(\partial_{0,\nu}E+A\right)^{-1}f\right)(\Im z))\in\mathcal{H}^{2}(\mathbb{C}_{\Re>\nu};H),
\]
due to the boundedness and analyticity of 
\[
\left(\mathbb{C}_{\Re>\nu}\ni z\mapsto z^{-k}\left(zE+A\right)^{-1}\in L(H)\right).
\]
This proves the claim.
\end{proof}
Corollary \ref{cor:solOp} states a particular boundedness property for
the solution operator associated with \prettyref{eq:IVP0}. This can
be made more precise by introducing a scale of extrapolation spaces
associated with $\partial_{0,\nu}$.
\begin{defn}
Let $k\in\mathbb{N}$, $\nu>0$. Then we define $H_{\nu}^{k}(\mathbb{R};H)\coloneqq D(\partial_{0,\nu}^{k})$
endowed with the scalar product $\langle\phi,\psi\rangle_{k}\coloneqq\langle\partial_{0,\nu}^{k}\phi,\partial_{0,\nu}^{k}\psi\rangle_{0}$.
Quite similarly, we define $H_{\nu}^{-k}(\mathbb{R};H)$ as the completion
of $L_{2,\nu}(\mathbb{R};H)$ with respect to $\langle\phi,\psi\rangle_{-k}\coloneqq\langle\partial_{0,\nu}^{-k}\phi,\partial_{0,\nu}^{-k}\psi\rangle_{0}$.
\end{defn}

We observe that the spaces $(H_{\nu}^{k}(\mathbb{R};H))_{k\in\mathbb{Z}}$
are nested in the sense that $j_{k\to\ell}\colon H_{\nu}^{k}(\mathbb{R};H)\hookrightarrow H_{\nu}^{\ell}(\mathbb{R};H),x\mapsto x$,
whenever $k\geq\ell$.
\begin{rem}
The operator $\partial_{0,\nu}^{\ell}$ can be considered as a densely
defined isometry from $H^{k}$ to $H^{k-\ell}$ with dense range for
all $k\in\mathbb{Z}$. The closure of this densely defined isometry
will be given the same name. In this way, we can state the boundedness
property of the solution operator in Corollary \ref{cor:solOp} equivalently
as follows:
\[
\left(\partial_{0,\nu}E+A\right)^{-1}\in L\left(L_{2,\nu}(\mathbb{R};H),H_{\nu}^{-k}(\mathbb{R};H)\right).
\]
More generally, as $\left(\partial_{0,\nu}E+A\right)^{-1}$ and $\partial_{0,\nu}^{-1}$
commute, we obtain 
\[
\left(\partial_{0,\nu}E+A\right)^{-1}\in L\left(H_{\nu}^{j}(\mathbb{R};H),H_{\nu}^{j-k}(\mathbb{R};H)\right)
\]
for each $j\in\mathbb{Z}$. 
\end{rem}

Note that by the Sobolev embedding theorem (see e.g.~\cite[Lemma 5.2]{KPSTW14_OD})
the $\delta$-distribution of point evaluation at $0$ is an element
of $H_{\nu}^{-1}(\mathbb{R};H)$; in fact it is the derivative of
$\chi_{\mathbb{R}_{\geq0}}\in L_{2,\nu}(\mathbb{R};H)=H_{\nu}^{0}(\mathbb{R};H)$.
With these preparations at hand, we consider the following implementation
of the initial value problem stated in \prettyref{eq:IVP0}: Let $u_{0}\in H$.
Find $u\in H_{\nu}^{-k}(\mathbb{R};H)$ such that
\begin{equation}
\left(\partial_{0,\nu}E+A\right)u=\delta\cdot Eu_{0}.\label{eq:IVPdis}
\end{equation}
\begin{thm}
\label{thm:solthdis}Let $(E,A)$ be regular. Then for all $u_{0}\in H$
there exists a unique $u\in H_{\nu}^{-k}(\mathbb{R};H)$ such that
\prettyref{eq:IVPdis} holds. Moreover, we have
\[
u=\chi_{\mathbb{R}_{\geq0}}u_{0}-\left(\partial_{0,\nu}E+A\right)^{-1}\chi_{\mathbb{R}_{\geq0}}Au_{0}
\]
and 
\[
\spt\partial_{0,\nu}^{-k}u\subseteq\mathbb{R}_{\geq0}.
\]
\end{thm}

\begin{proof}
Note that the unique solution is given by 
\[
u=\left(\partial_{0,\nu}E+A\right)^{-1}\delta\cdot Eu_{0}\in H_{\nu}^{-k-1}(\mathbb{R};H).
\]
Hence, 
\begin{align*}
u-\chi_{\mathbb{R}_{\geq0}}u_{0} & =\left(\partial_{0,\nu}E+A\right)^{-1}\left(\delta\cdot Eu_{0}-\left(\partial_{0,\nu}E+A\right)\chi_{\mathbb{R}_{\geq0}}u_{0}\right)\\
 & =-\left(\partial_{0,\nu}E+A\right)^{-1}\chi_{\mathbb{R}_{\geq0}}Au_{0},
\end{align*}
which shows the desired formula. Since $\chi_{\mathbb{R}_{\geq0}}u_{0}\in L_{2,\nu}(\mathbb{R};H)\hookrightarrow H_{\nu}^{-k}(\mathbb{R};H)$
and $\left(\partial_{0,\nu}E+A\right)^{-1}\chi_{\mathbb{R}_{\geq0}}Au_{0}\in H_{\nu}^{-k}(\mathbb{R};H)$
by Corollary \ref{cor:solOp} we obtain the asserted regularity for $u$.
The support statement follows from the causality statement in
Corollary \ref{cor:solOp}. 
\end{proof}
In the concluding section, we will discuss the spaces $\mathrm{IV}_{k}$
in connection to $(E,A)$ and will prove the main results of this
paper mentioned in the introduction.

\section{Proofs of the main results and initial value spaces}

Again, we assume that $H$ is a Hilbert space, and that $E,A\in L(H)$
with $(E,A)$ regular. 

At first, we turn to the proof of Proposition \ref{prop:stabInde}. For
this, we note some elementary consequences of the definition of $\mathrm{IV}_{k}$
and of regularity.
\begin{lem}
\label{lem:aux}\begin{enumerate}[(a)]

\item For all $k\in\mathbb{N}$, we have $\mathrm{IV}_{k}\supseteq\mathrm{IV}_{k+1}.$

\item Let $s\in\mathbb{C}\cap\rho(E,A)$. Then
\[
E(sE+A)^{-1}A=A(sE+A)^{-1}E.
\]
\item Let $k\in\mathbb{N}$, $x\in\mathrm{IV}_{k}$. Then for all
$s\in\mathbb{C}\cap\rho(E,A)$ we have
\[
(sE+A)^{-1}Ex\in\mathrm{IV}_{k+1}.
\]

\item Let $s\in\mathbb{C}\cap\rho(E,A)\setminus\{0\}$. Then
\[
(sE+A)^{-1}E=\frac{1}{s}-\frac{1}{s}(sE+A)^{-1}A.
\]

\item Let $k\in\mathbb{N}$, $x\in\mathrm{IV}_{k}$. Then for all
$s\in\mathbb{C}\cap\rho(E,A)\setminus\{0\}$ we have
\[
(sE+A)^{-1}Ex=\frac{1}{s}x+\sum_{\ell=1}^{k}\frac{1}{s^{\ell+1}}x_{\ell}+\frac{1}{s^{k+1}}(sE+A)^{-1}Aw.
\]
for some $w\in H$, $x_{1},\ldots,x_{k}\in H$.

\end{enumerate}
\end{lem}

\begin{proof}
The proof of (a) is an induction argument. The claim is trivial for
$k=0$. For the inductive step, we see that the assertion follows
using the induction hypothesis by 
\[
\mathrm{IV}_{k+1}=A^{-1}[E[\mathrm{IV}_{k}]]\supseteq A^{-1}[E[\mathrm{IV}_{k+1}]]=\mathrm{IV}_{k+2}.
\]
Next, we prove (b). We compute
\begin{align*}
E(sE+A)^{-1}A= & E(sE+A)^{-1}(sE+A-sE)\\
= & E-E(sE+A)^{-1}sE\\
= & E-\left(sE+A-A\right)(sE+A)^{-1}E\\
= & A(sE+A)^{-1}E.
\end{align*}
We prove (c), by induction on $k$. For $k=0$, we let $x\in\mathrm{IV}_{0}=H$
and put $y\coloneqq\left(sE+A\right)^{-1}Ex.$ Then, by (b), we get
that
\[
Ay=A\left(sE+A\right)^{-1}Ex=E\left(sE+A\right)^{-1}Ax\in R(E)=E[\mathrm{IV}_{0}].
\]
Hence, $y\in\mathrm{IV}_{1}$. For the inductive step, we assume that
the assertion holds for some $k\in\mathbb{N}$. Let $x\in\mathrm{IV}_{k+1}$.
We need to show that $y\coloneqq\left(sE+A\right)^{-1}Ex\in\mathrm{IV}_{k+2}$.
For this, note that there exists $w\in\mathrm{IV}_{k}$ such that
$Ax=Ew$. In particular, by the induction hypothesis, we have $\left(sE+A\right)^{-1}Ew\in\mathrm{IV}_{k+1}$.
Then we compute using (b) again,
\begin{align*}
Ay & =A\left(sE+A\right)^{-1}Ex\\
 & =E\left(sE+A\right)^{-1}Ax\\
 & =E\left(sE+A\right)^{-1}Ew\in E[\mathrm{IV}_{k+1}].
\end{align*}
Hence, $y\in\mathrm{IV}_{k+2}$ and (c) is proved.\\
For (d), it suffices to observe
\begin{align*}
(sE+A)^{-1}E&=\frac{1}{s}(sE+A)^{-1}sE\\
&=\frac{1}{s}(sE+A)^{-1}(sE+A-A)\\
&=\frac{1}{s}-\frac{1}{s}(sE+A)^{-1}A.
\end{align*}
In order to prove part (e), we proceed by induction on $k\in\mathbb{N}$.
The case $k=0$ has been dealt with in part (d) by choosing $w=-x.$
For the inductive step, we let $x\in\mathrm{IV}_{k+1}$. By definition
of $\mathrm{IV}_{k+1}$, we find $y\in\mathrm{IV}_{k}$ such that
$Ax=Ey.$ By induction hypothesis, we find $w\in H$ and $x_{1},\ldots,x_{k}\in H$
such that 
\[
(sE+A)^{-1}Ey=\frac{1}{s}y+\sum_{\ell=1}^{k}\frac{1}{s^{\ell+1}}x_{\ell}+\frac{1}{s^{k+1}}(sE+A)^{-1}Aw.
\]
With this we compute using (d)
\begin{align*}
(sE+A)^{-1}Ex & =\frac{1}{s}x-\frac{1}{s}(sE+A)^{-1}Ax\\
 & =\frac{1}{s}x-\frac{1}{s}(sE+A)^{-1}Ey\\
 & =\frac{1}{s}x-\frac{1}{s}\left(\frac{1}{s}y+\sum_{\ell=1}^{k}\frac{1}{s^{\ell+1}}x_{\ell}+\frac{1}{s^{k+1}}(sE+A)^{-1}Aw\right)\\
 & =\frac{1}{s}x+\sum_{\ell=1}^{k+1}\frac{1}{s^{\ell+1}}\tilde{x}_{\ell}+\frac{1}{s^{k+2}}(sE+A)^{-1}A\tilde{w},
\end{align*}
with $\tilde{x}_{1}=-y,\tilde{x}_{\ell}=-x_{\ell-1}$ for $\ell\geq2$
and $\tilde{w}=-w$. 
\end{proof}
With \prettyref{lem:aux}(a), we obtain the following reformulation
of Proposition \ref{prop:stabInde}.
\begin{prop}
\label{prop:stabInd0}Assume that $E[\mathrm{IV}_{\ind(E,A)}]\subseteq H$
is closed. Then 
\[\mathrm{IV}_{\ind(E,A)+1}\subseteq\mathrm{IV}_{\ind(E,A)+2}.\]
\end{prop}

\begin{proof}
Note that the closedness of $E[\mathrm{IV}_{\ind(E,A)}]$ implies
the same for the space $\mathrm{IV}_{\ind(E,A)+1}$ since $A$ is continuous.
We set $k\coloneqq\ind(E,A)$. Let $x\in\mathrm{IV}_{k+1}$. Then
we need to find $y\in\mathrm{IV}_{k+1}$ with $Ax=Ey$. By definition
there exists $x_{0}\in\mathrm{IV}_{k}$ with the property $Ax=Ex_{0}$.
For $n\in\mathbb{N}$ large enough we define $y_{n}\coloneqq n\left(nE+A\right)^{-1}Ex_{0}.$
Since, $x_{0}\in\mathrm{IV}_{k}$, we deduce with \prettyref{lem:aux}(c)
that $y_{n}\in\mathrm{IV}_{k+1}$. Moreover, by \prettyref{lem:aux}(e),
$(y_{n})_{n}$ is bounded. Choosing a suitable subsequence for which
we use the same name, we may assume that $(y_{n})_{n}$ is weakly
convergent to some $y\in H$. The closedness of $\mathrm{IV}_{k+1}$
implies $y\in\mathrm{IV}_{k+1}$. Then using \prettyref{lem:aux}(e)
we find $w\in H$ and $x_{1},\ldots,x_{k+1}\in H$ such that 
\[
\left(nE+A\right)^{-1}Ex_{0}=\sum_{\ell=0}^{k}\frac{1}{n^{\ell+1}}x_{\ell}+\frac{1}{n^{k+1}}(nE+A)^{-1}Aw.
\]
Hence, we obtain
\begin{align*}
Ey & =\wlim_{n\to\infty}Ey_{n}\\
 & =\wlim_{n\to\infty}E\left(nE+A\right)^{-1}nEx_{0}\\
 & =\wlim_{n\to\infty}nE\left(nE+A\right)^{-1}Ax\\
 & =\wlim_{n\to\infty}\left(nE+A-A\right)\left(nE+A\right)^{-1}Ax\\
 & =Ax-\wlim_{n\to\infty}A\left(nE+A\right)^{-1}Ex_{0}\\
 & =Ax-A\wlim_{n\to\infty}\left(\sum_{\ell=0}^{k}\frac{1}{n^{\ell+1}}x_{\ell}+\frac{1}{n^{k+1}}(nE+A)^{-1}Aw\right)=Ax,
\end{align*}
which yields the assertion.
\end{proof}
With an idea similar to the one in the proof of Proposition \ref{prop:stabInde}
(Proposition \ref{prop:stabInd0}), it is possible to show that $E\colon\mathrm{IV}_{k+1}\to E[\mathrm{IV}_{k}]$
is an isomorphism if $k=\ind(E,A)$ and $E[\mathrm{IV}_{k}]\subseteq H$
is closed. We will need this result also in the proof of our main
theorem.
\begin{thm}
\label{thm:Eiso}Let $(E,A)$ be regular and assume that $E[\mathrm{IV}_{k}]\subseteq H$
is closed, $k=\ind(E,A).$ Then
\[
E\colon\mathrm{IV}_{k+1}\to E[\mathrm{IV}_{k}],x\mapsto Ex
\]
is a Banach space isomorphism.
\end{thm}

\begin{proof}
Note that by the closed graph theorem, it suffices to show that the
operator under consideration is one-to-one and onto. So, for proving
injectivity, we let $x\in\mathrm{IV}_{k+1}$ such that $Ex=0.$ By
definition, there exists $y\in\mathrm{IV}_{k}$ such that $Ey=Ax=Ax+nEx$
for all $n\in\mathbb{N}$. Hence, for $n\in\mathbb{N}$ large enough,
we have $x=\left(nE+A\right)^{-1}Ey$. Thus, from $y\in\mathrm{IV}_{k}$
we deduce with the help of \prettyref{lem:aux}(e) that there exist
$w,x_1,\ldots.x_{k}\in H$ such that
\[
x=\left(nE+A\right)^{-1}Ey=\frac{1}{n}y+\sum_{\ell=1}^{k}\frac{1}{n^{\ell+1}}x_{\ell}+\frac{1}{n^{k+1}}(nE+A)^{-1}Aw\to0\quad(n\to\infty),
\]
which shows $x=0$. 

Next, let $y\in E[\mathrm{IV}_{k}]$. For large enough $n\in\mathbb{N}$
we put 
\[
w_{n}\coloneqq(nE+A)^{-1}ny.
\]
By \prettyref{lem:aux}(c), we obtain that $w_{n}\in\mathrm{IV}_{k+1}.$
Let $x\in\mathrm{IV}_{k}$ with $Ex=y$. Then, using \prettyref{lem:aux}(e),
we find $w,x_{1},\ldots,x_{k}\in H$ such that
\begin{align*}
w_{n} & =(nE+A)^{-1}ny\\
 & =(nE+A)^{-1}nEx\\
 & =x+\sum_{\ell=1}^{k}\frac{1}{n^{\ell}}x_{\ell}+\frac{1}{n^{k}}(nE+A)^{-1}Aw,
\end{align*}
proving the boundedness of $(w_{n})_{n}.$ Without loss of generality,
we may assume that $(w_{n})_{n}$ weakly converges to $z\in\mathrm{IV}_{k+1}=A^{-1}[E[\mathrm{IV}_{k}]]$.
Hence, 
\begin{align*}
Ez&=\wlim_{n\to\infty}Ew_{n}\\
&=\wlim_{n\to\infty}\frac{1}{n}\left(nE+A\right)w_{n}\\
&=\wlim_{n\to\infty}\frac{1}{n}\left(nE+A\right)(nE+A)^{-1}ny\\
&=y\tag*{{\qedhere}}.
\end{align*}
\end{proof}
Next, we come to the proof of our main result \prettyref{thm:mainTh}, which we restate here for convenience.
\begin{thm}
\label{thm:mainTh-1}Assume that $E[\mathrm{IV}_{\ind(E,A)}]\subseteq H$
is closed, $u_{0}\in\mathrm{IV}_{\ind(E,A)+1}$. Then \prettyref{eq:IVPdis}
has a unique continuously differentiable solution $u:\mathbb{R}_{>0}\to H$,
satisfying $u(0+)=u_{0}$ and 
\begin{equation}
Eu'(t)+Au(t)=0\quad(t>0).\label{eq:dAe}
\end{equation}
Moreover, the solution coincides with the solution given in \prettyref{thm:solthdis}.
\end{thm}

\begin{proof}
Let $u_{0}\in\mathrm{IV}_{\ind(E,A)+1}$. We denote $\tilde{E}\colon\mathrm{IV}_{k+1}\to E[\mathrm{IV}_{k}],x\mapsto Ex,$
where $k=\ind(E,A)$. By \prettyref{thm:Eiso}, we have that $\tilde{E}$
is an isomorphism. For $t>0$, we define 
\[
u(t)\coloneqq\exp\left(-t\tilde{E}^{-1}A\right)u_{0}.
\]
Then $u(0+)=u_{0}.$ Moreover, $u(t)$ is well-defined. Indeed, if
$u_{0}\in\mathrm{IV}_{k+1}$ then $Au_{0}\in E[\mathrm{IV}_{k}]$.
Hence, $\tilde{E}^{-1}Au_{0}\in\mathrm{IV}_{k+1}$ is well-defined.
Since $E[\mathrm{IV}_{k}]$ is closed, and $A$ is continuous, we
infer that $\mathrm{IV}_{k+1}$ is a Hilbert space. Thus, we deduce
that $u\colon\mathbb{R}_{>0}\to\mathrm{IV}_{k+1}$ is continuously
differentiable. In particular, we obtain
\[
\mathrm{IV}_{k+1}\ni u'(t)=-\tilde{E}^{-1}Au(t).
\]
If we apply $\tilde{E}$ to both sides of the equality, we obtain
\prettyref{eq:dAe}. If $u:\mathbb{R}_{>0}\to H$ is a continuously
differentiable solution of \prettyref{eq:dAe} with $u(0+)=u_{0},$
we infer that $u\in L_{2,\nu}(\mathbb{R};H)$ for some $\nu>0$
large enough, where we extend $u$ to $\mathbb{R}_{<0}$ by zero.
Hence, 
\[
\partial_{0,\nu}Eu+Au=E\partial_{0,\nu}u+Au=Eu'+Au+\delta\cdot Eu(0+)=\delta\cdot Eu_{0},
\]
where we have used that $u$ is differentiable on $\mathbb{R}_{<0}\cup\mathbb{R}_{>0}$
and jumps at $0$. Thus, $u$ is the solution given in \prettyref{thm:solthdis},
from which we also derive the uniqueness. 
\end{proof}

We conclude with a comment on the proof of Corollary \ref{cor:mainCor}.
\begin{rem}
We note that the condition $u_{0}\in\mathrm{IV}_{\ind(E,A)+1}$ arises
naturally if we assume that $\mathrm{IV}_{j}$ is closed for each
$j\in\mathbb{N}.$ Indeed, if $u:\mathbb{R}_{>0}\to H$ is a continuously
differentiable solution of \prettyref{eq:dAe}, we infer that 
\[
Au(t)=-Eu'(t)\quad(t>0)
\]
and thus $u(t)\in\mathrm{IV}_{1}$ for $t>0$. Since $\mathrm{IV}_{1}$
is closed, we derive $u'(t)\in\mathrm{IV}_{1}$ and hence, inductively
$u(t)\in\bigcap_{j\in\mathbb{N}}\mathrm{IV}_{j}$ for each $t>0.$
Since $\bigcap_{j\in\mathbb{N}}\mathrm{IV}_{j}=\mathrm{IV}_{\ind(E,A)+1}$
by Proposition \ref{prop:stabInd0}, we get 
\[
u_{0}=u(0+)\in\mathrm{IV}_{\ind(E,A)+1}.
\]
\end{rem}


\begin{thebibliography}{1}

\bibitem{Berger2012}
T.~Berger, A.~Ilchmann, and S.~Trenn.
\newblock The quasi-{W}eierstra\ss \ form for regular matrix pencils.
\newblock {\em Linear Algebra Appl.}, 436(10):4052--4069, 2012.

\bibitem{Dai1989}
L.~Dai.
\newblock {\em Singular Control Systems}.
\newblock Springer-Verlag New York, Inc., Secaucus, NJ, USA, 1989.

\bibitem{KPSTW14_OD}
A.~Kalauch, R.~Picard, S.~Siegmund, S.~Trostorff, and M.~M.~Waurick.
\newblock {A Hilbert space perspective on ordinary differential equations with
  memory term}.
\newblock {\em {Journal of Dynamics and Differential Equations}},
  26(2):369--399, 2014.

\bibitem{Mehrmann2006}
P.~{Kunkel} and V.~{Mehrmann}.
\newblock {\em {Differential-Algebraic Equations. Analysis and Numerical
  Solution.}}
\newblock European Mathematical Society Publishing House, Z\"urich, 2006.

\bibitem{Picard1989}
R.~Picard.
\newblock {\em Hilbert Space Approach to Some Classical Transforms.} Volume 196
  of Pitman Research Notes in Mathematics Series.
\newblock Longman Scientific \& Technical, Harlow; copublished in the United
  States with John Wiley \& Sons, Inc., New York, 1989.

\bibitem{PicPhy}
R.~Picard.
\newblock {A structural observation for linear material laws in classical
  mathematical physics}.
\newblock {\em {Mathematical Methods in the Applied Sciences}}, 32:1768--1803,
  2009.

\bibitem{rudin1987real}
W.~Rudin.
\newblock {\em {Real and Complex Analysis}}.
\newblock Mathematics series. McGraw-Hill, 1987.

\bibitem{Trostorff2017b}
S.~Trostorff and M.~Waurick.
\newblock {On differential-algebraic equations in infinite dimensions}.
\newblock Technical report, TU Dresden, University of Strathclyde, 2017.

\end{thebibliography}
\end{document}